\documentclass[]{amsart}

\usepackage[english]{babel}
\usepackage[T1]{fontenc}

\usepackage[a4paper]{geometry}
\usepackage{amsmath}
\usepackage{verbatim}
\usepackage[colorinlistoftodos]{todonotes}
\usepackage{mathtools} 

\usepackage[all]{xy}
\usepackage{amssymb,mathrsfs, amscd, graphicx, array, multirow,
enumerate, amsfonts, euscript}
\usepackage{comment}
\usepackage{upgreek} 
\usepackage{float}
\usepackage{caption, subcaption}
\usepackage{enumitem}

\usepackage{marvosym}
\xyoption{poly}
\usepackage{url}

\allowdisplaybreaks
\usepackage{color}
\usepackage{xcolor}

\usepackage{tikz-cd}
\usepackage{tikz}
\usetikzlibrary{matrix,arrows}

\usepackage{mathrsfs}

\theoremstyle{plain}
\newtheorem{thm}{Theorem}[section]
\newtheorem{lemma}[thm]{Lemma}

\newtheorem{prop}[thm]{Proposition}
\newtheorem{cor}[thm]{Corollary}

\theoremstyle{definition}

\newtheorem{defn}[thm]{Definition}
\newtheorem{Example}[thm]{Example}

\theoremstyle{remark}
\newtheorem{remark}[thm]{Remark}

\numberwithin{equation}{section}

\def\Hom{\mathrm{Hom}}

\def\makeop#1{\expandafter\def\csname#1\endcsname
  {\mathop{\rm #1}\nolimits}\ignorespaces}
\makeop{Hom}   \makeop{End}   \makeop{Aut}   \makeop{Isom}  \makeop{Pic} 
\makeop{Gal}   \makeop{ord}   \makeop{Char}  \makeop{Div}   \makeop{Lie} 
\makeop{PGL}   \makeop{Corr}  \makeop{PSL}   \makeop{sgn}   \makeop{Spf}
\makeop{Spec}  \makeop{Tr}    \makeop{Nr}    \makeop{Fr}    \makeop{disc}
\makeop{Proj}  \makeop{supp}  \makeop{ker}   \makeop{im}    \makeop{dom}
\makeop{coker} \makeop{Stab}  \makeop{SO}    \makeop{SL}    \makeop{SL}
\makeop{Cl}    \makeop{cond}  \makeop{Br}    \makeop{inv}   \makeop{rank}
\makeop{id}    \makeop{Fil}   \makeop{Frac}  \makeop{GL}    \makeop{SU}
\makeop{Nrd}   \makeop{Sp}    \makeop{Tr}    \makeop{Trd}   \makeop{diag}
\makeop{Res}   \makeop{ind}   \makeop{depth} \makeop{Tr}    \makeop{st}
\makeop{Ad}    \makeop{Int}   \makeop{tr}    \makeop{Sym}   \makeop{can}
\makeop{length}\makeop{SO}    \makeop{torsion} \makeop{GSp} \makeop{Ker}
\makeop{Adm}   \makeop{Mat}   \makeop{Cor}   \makeop{Inf} \makeop{Cok}
\makeop{Ver}
\def\makebb#1{\expandafter\def
  \csname bb#1\endcsname{{\mathbb{#1}}}\ignorespaces}
\def\makebf#1{\expandafter\def\csname bf#1\endcsname{{\bf
      #1}}\ignorespaces} 
\def\makegr#1{\expandafter\def
  \csname gr#1\endcsname{{\mathfrak{#1}}}\ignorespaces}
\def\makescr#1{\expandafter\def
  \csname scr#1\endcsname{{\EuScript{#1}}}\ignorespaces}
\def\makecal#1{\expandafter\def\csname cal#1\endcsname{{\mathcal
      #1}}\ignorespaces} 

\def\doLetters#1{#1A #1B #1C #1D #1E #1F #1G #1H #1I #1J #1K #1L #1M
                 #1N #1O #1P #1Q #1R #1S #1T #1U #1V #1W #1X #1Y #1Z}
\def\doletters#1{#1a #1b #1c #1d #1e #1f #1g #1h #1i #1j #1k #1l #1m
                 #1n #1o #1p #1q #1r #1s #1t #1u #1v #1w #1x #1y #1z}
\doLetters\makebb   \doLetters\makecal  \doLetters\makebf
\doLetters\makescr 
\doletters\makebf   \doLetters\makegr   \doletters\makegr
     \def\qed{\qedmark\medbreak}%
\def\qedmark{{\enspace\vrule height 6pt width 5pt depth 1.5pt}}%
    \def\setminus{\smallsetminus}

\def\Zp{{\bbZ}_p}
\def\Qbar{\overline{\bbQ}}

\def\wh{\widehat}

\def\Ind{{\rm Ind}}

\def\G{\mathbb{G}}

\def\Q{\mathbb{Q}}
\def\Z{\mathbb{Z}}
\def\A{\mathbb{A}}
\def\C{\mathbb{C}}

\def\char{\text{char }}

\def\embed{\hookrightarrow}
\def\F{\bbF}
\def\ol{\overline}

\newcommand{\<}{\langle}   
\renewcommand{\>}{\rangle} 

\newcommand{\isoto}{\stackrel{\sim}{\to}}

\makeatletter
\newcommand{\xdashrightarrow}[2][]
  {\ext@arrow 0359\rightarrowfill@@{#1}{#2}}
\newcommand{\xdashleftarrow}[2][]
  {\ext@arrow 3095\leftarrowfill@@{#1}{#2}}
\newcommand{\xdashleftrightarrow}[2][]{\ext@arrow 3359\leftrightarrowfill@@{#1}{#2}}
\def\rightarrowfill@@{\arrowfill@@\relax\relbar\rightarrow}
\def\leftarrowfill@@{\arrowfill@@\leftarrow\relbar\relax}
\def\leftrightarrowfill@@{\arrowfill@@\leftarrow\relbar\rightarrow}
\def\arrowfill@@#1#2#3#4{%
  $\m@th\thickmuskip0mu\medmuskip\thickmuskip\thinmuskip\thickmuskip
   \relax#4#1
   \xleaders\hbox{$#4#2$}\hfill
   #3$%
}
\makeatother

\usepackage[OT2,T1]{fontenc}
\DeclareSymbolFont{cyrletters}{OT2}{wncyr}{m}{n}
\DeclareMathSymbol{\Sha}{\mathalpha}{cyrletters}{"58}

\def\Gmk{\G_{{\rm m}, k}}

\begin{document}

\title[]{Computing Tate-Shafarevich groups of multinorm one tori of Kummer type}

\author{Jun-Hao Huang}
\address{(Huang) Department of Mathematics, National Taiwan Normal University, Taipei, Taiwan, 116059}
\email{junhao20150115@gmail.com}

\author{Fan-Yun Hung}
\address{(Hung) Institute of Mathematics, Academia Sinica, 
Taipei, Taiwan, 10617}
\email{fanyunhung@gate.sinica.edu.tw}

\author{Pei-Xin Liang}
\address{(Liang) Department of Mathematics, National Tsing-Hua University, Hsin-Chu, Taiwan,  300044}
\email{cindy11420@gmail.com}

\author{Chia-Fu Yu}
\address{(Yu) Institute of Mathematics, Academia Sinica and the National Center for Theoretical Sciences,
Taipei, Taiwan, 10617}
\email{chiafu@math.sinica.edu.tw}


\date{\today}
\subjclass[2010]{11G35, 12G05.} 
\keywords{Multinorm principles, Tate-Shafarevich groups, Multinorm one tori. }  

\maketitle

\begin{abstract}
A multinorm one torus associated to a commutative  \'etale algebra $L$ over a global field $k$ is of Kummer type if each factor of $L$ is a cyclic Kummer extension. In this paper we compute the Tate-Shafarevich group of such tori based on recent works of Bayer-Fluckiger, T.-Y. Lee and Parimala, and of T.-Y.~Lee. We also implement an effective algorithm using SAGE which computes the Tate-Shafarevich groups when each factor of $L$ is contained in a fixed concrete bicyclic extension of $k$.
\end{abstract}

\section{Introduction}
\label{sec:I}

Let $k$ be a global field and let $L=\prod_{i=0}^m K_i$ be a product of finite separable field extensions $K_i$ of $k$. The norm map $N_{L/k}$ from $L$ to $k$ is defined by $N_{L/k}(x):=\prod_{i} N_{K_i/k} (x_i)$ for $x=(x_i)\in L$. 
Let $\A_k$ denote the adele ring of $k$ and $\A_L:=L\otimes_k \A_k=\prod_{i=0}^m \A_{K_i}$ the adele ring of $L$. We have the norm map $N_{L/k}:\A_L^\times \to \A_k^\times$, sending $(x_i)$ to $\prod_{i} N_{K_i/k}(x_i)$. We say that \emph{the multinorm principle} holds for $L/k$ if 
\begin{equation}\label{eq:I.1}
    k^\times \cap N_{L/k}(\A_L^\times)=N_{L/k}(L^\times).
\end{equation}
The quotient group
\begin{equation}\label{eq:I.2}
    \Sha(L/k):=\frac{k^\times \cap N_{L/k}(\A_L^\times)}{N_{L/k}(L^\times)}
\end{equation}
is called the Tate-Shafarevich group of $L/k$, which measures the deviation of the validity of the multinorm principle. 

H\"urlimann~\cite[Proposition~3.3]{hurlimann} showed that the multinorm principle holds for $L=K_0\times K_1$ provided that one of $K_i$ is cyclic and the other is Galois (the second condition is actually superfluous as later proved by \cite[Proposition~4.1]{BLP19}. Pollio and Rapinchuk~\cite[Theorem, p.~803]{pollio-rapinchuk} showed the case when the Galois closures of $K_0$ and $K_1$ are linearly disjoint; the former author proved \cite[Theorem 1]{pollio} that if $K_0$ and $K_1$ are abelian extensions of $k$, then $\Sha((K_0\times K_1)/k)=\Sha((K_0 \cap K_1)/k)$. 
Demarche and D.~Wei~\cite{demarche-wei} constructed a family of examples, showing that the equality $\Sha((K_0\times K_1)/k)= \Sha((K_0 \cap K_1)/k)$ is no longer true when $K_0$ and $K_1$ are non-abelian Galois extensions. 
Bayer-Fluckiger, T.-Y. Lee and Parimala~\cite{BLP19} studied the Tate-Shafarevich group of
general multinorm one tori in which $K_0$ is a cyclic extension. 
Among others, they computed $\Sha(L/k)$ in the case of products of extensions of prime degree $p$. 
Very recently T.-Y.~Lee~\cite{lee2022tate} computed explicitly $\Sha(L/k)$ for the cases where every factor is cyclic of degree $p$-power,  and by the reduction result of \cite{BLP19}, of arbitrary degree. 
The study of the multinorm principle is also inspired by the work of 
Prasad and Rapinchuk \cite{prasad-rapinchuk}, where they settled the problem of the local-global principle for embeddings of fields with involution into simple algebras with involution. \\

The aim of this article is to compute more examples of the Tate-Shafarevich groups of multiple norm one tori based on T.-Y. Lee's general formulas. We consider the étale $k$-algebras $L=\prod_i K_i$, where each $K_i/k$ is a cyclic extension of $p$-power degree and $k$ contains the roots of unity of sufficiently large degree of $p$ (prime to the characteristic of $k$). The idea is to translate all invariants in Lee's formulas from the number-theoretic description into a combinatorial one. This allows us to compute the Tate-Sharafevich groups much more effectively. The main reason is that computing the decomposition groups of a Galois extension of large degree by computer is extremely time-consuming. 
We implement an algorithm using SageMath which computes $\Sha(L)$ for input data where $k=\Q(\zeta_{p^n})$ is the $p^n$-th cyclotomic field and  $K_i$ are cyclic subextensions of a bicyclic extension $k(\ell_1^{1/p^n}, \ell_2^{1/p^n})$ with primes $\ell_1,\ell_2\neq p$, subject to the condition $\cap_{i=0}^m K_i=k$. As our algorithm is a based on a combinatorial description, the computing time does not take much longer when $p$ and $n$ are large. 

This paper is organized as follows. 
In Section~\ref{sec:S}, we organize several results of Bayer-Fluckiger, T.-Y.~Lee and Parimala and describe
formulas for the Tate-Shafarevich groups due to T.-Y.~Lee. Section~\ref{sec:R} discusses the assumptions in Theorem~\ref{thm:Lee}. In Section~\ref{sec:K} we translate all invariants in Lee's formulas from the number-theoretic description into a combinatorial one in the case where $k=\Q(\zeta_{p^n})$ is the $p^n$-th cyclomotic field. Section~\ref{sec:D} computes the decomposition groups of any subfield extension of the aforementioned bicyclic extension $k(\ell_1^{1/p^n}, \ell_2^{1/p^n})$. Putting all together in the last section, we compute the Tate-Shafarevich group of the multinorm one torus in questions and show examples.

\section{The Tate-Shafarevich groups of multinorm one tori}\label{sec:S}
In this section, we organize several results from Bayer-Fluckiger--Lee--Parimala  \cite{BLP19} and describe formulas for the Tate-Shafarevich groups of multinorm one tori due to T.-Y. Lee \cite{lee2022tate}.

\subsection{}\label{sec:S.1} Let $k$ be a global field and $k_s$ be a separable field of $k$ whose Galois group is denoted by $\Gamma_k$. Let $\Omega_k$ be the set of all places of $k$.
Let $T$ be an algebraic torus over $k$. Denote by $\widehat{T}:=\textrm{Hom}_{k_s}(T,\mathbb{G}_m)$ be the character group of $T$; it is a finite free $\Z$-module with a continuous action of $\Gamma_k$. Let $H^i(k,\widehat{T})$ denote the $i$-th Galois cohomology group of $\Gamma_k$ with coefficients in $\widehat{T}$.  
\begin{defn}
The $i$-th \textit{Tate-Shafarevich group} and \textit{algebraic Tate-Shafarevich group} of $\wh T$ are defined by
\[
        \Sha^{i}(k,\widehat{T}):={\rm{Ker}}\left(H^{i}(k,\widehat{T})\rightarrow\prod\limits_{v\in\Omega_{k}}H^{i}(k_{v},\widehat{T})\right)
\]
and
\[
        \Sha_{\omega}^{i}(k,\widehat{T}) := \left\{[C]\in H^{i}(k,\widehat{T})|\text{ }[C]_{v}=0\text{\rm{ for almost all }} v\in\Omega_{k}\right\},
        \]
respectively, where $[C]_v$ is the class of $[C]$ in $H^i(k_v, \wh T)$ under the restriction map $H^i(k, \wh T) \to H^i(k_v, \wh T)$.         
\end{defn}

Let $L=\prod_{i=0}^m K_i$ be an étale algebra over $k$, where each $K_i$ is a cyclic extension of $k$ of degree degree $d_i$ in $k_s$. Let $N_{L/k}: R_{L/k} \G_{m,L}\rightarrow\Gmk$ be the norm morphism, and let $T_{L/k}=\Ker N_{L/k}$ be the multinorm one torus associated to $L/k$. Put $\mathcal{I}=\{1,\ldots,m\}$ and  $K':=\prod_{i\in \mathcal{I}} K_i$.

Let $E:=K_0\otimes_k K'=\prod_{i\in \mathcal{I}}E_i$, where $E_i:=K_0 \otimes_k K_i$. We may regard the $k$-\'etale algebra  $E$ as an \'etale algebra over $K_0$ or over $K'$. 
Let $N_{E/K_0}$ and $N_{E/K'}$ be the norm maps from ${R}_{E/k}\mathbb{G}_{m,E}$ to itself, and define a morphism $f:{R}_{E/k}\mathbb{G}_{m,E} \to R_{L/k}\mathbb{G}_{m, L}$ by  $f(x)=(N_{E/K_0}(x)^{-1},N_{E/K'}(x))$. One easily checks that the image of $f$ is equal to $T_{L/k}$.
Let \textbf{$S_{K_0,K'}$} be the $k$-torus defined by the following exact sequence 
\begin{equation}\label{eq:SKK'.1}
  \begin{CD}
1 \longrightarrow S_{K_0,K'} \longrightarrow {{R}}_{E/k}(\mathbb{G}_{m,E}) \stackrel{f}\longrightarrow T_{L/k} \longrightarrow 1.
\end{CD}  
\end{equation}
The algebraic torus $S_{K_0,K'}$ also fits in the following exact sequence
\begin{equation}\label{eq:SKK'.2}
  \begin{CD}
1 \longrightarrow S_{K_0,K'} \longrightarrow R_{K'/k}(T_{E/K'})  @>N_{E/K_0}>> T_{K_0/k} \longrightarrow 1.
\end{CD}  
\end{equation}
Here $R_{K'/k}(T_{E/K'})=\prod_{i\in \mathcal{I}} R_{K_i/k} (T_{E_i/K_i})$.

\begin{prop}\cite[Lemma 3.1]{BLP19}\label{prop:H1SKK'}
There is a functorial natural isomorphism 
\begin{equation}\label{eq:H1SKK'}
    \Sha^1 (k, \wh S_{K_0,K'})\simeq \Sha^2(k, \wh T_{L/k}).
\end{equation}
\end{prop}
It follows from \eqref{eq:SKK'.1} that there is a natural isomorphism $H^1(k,T_{L/k})\simeq H^2(k, S_{K_0,K'})$. Then Proposition~\ref{prop:H1SKK'} follows from the Poitou-Tate duality. 

Define  
\[ \Sha(L):=\Sha^2(k, \wh T_{L/k}) \quad \text{and}\quad \Sha_\omega(L):=\Sha^2_\omega(k, \wh T_{L/k}). \]





For any prime number $p$ and any cyclic extension $M$ of $k$, let $M(p)$ denote the largest subfield of $M$ such that $[M(p):k]$ is a power of $p$. Also, if $p$ divides $[M:k]$, we denote by $M(p)_{\rm prim}$ the unique subfield of $M(p)$ of degree $p$ over $k$.

\begin{prop}
\cite[Propositions 5.16 and 8.6]{BLP19}\label{prop:BLP.1} 
Let $L=\prod_{i=0}^m K_i$ be a product of cyclic extensions of degree $d_i$ over $k$. Set $L(p):=K_0(p)\times K'$ and $L^*(p):=\prod_{i=0}^m K_i(p)$. Then we have isomorphisms
\[\Sha(L)=\mathop{\oplus}\limits_{p|d_0}\Sha(L(p)),\]
and 
\[\Sha(L)=\mathop{\oplus}\limits_{p|d}\Sha(L^*(p)),\]
where $d=\gcd(d_0,\dots, d_m)$. 
\end{prop}

Note that if $K_i(p)=k$ for some $i$, then $\Sha(L^*(p))=0$. Thus, if $p\nmid d$, then $\Sha(L^*(p))=0$.

\begin{thm}\label{thm:BLP.2} Let the notation be as Proposition~\ref{prop:BLP.1}. Assume that the field extensions $K_i(p)$ are linearly disjoint over $k$. Then 
\[\Sha(L^*(p))=0\Longleftrightarrow  \Sha(L^*(p)_{\rm prim})=0,\]
where $L^*(p)_{\rm prim}:=\prod_{i=0}^{m}K_i(p)_{\rm prim}.$
\end{thm}

\begin{proof}
This is \cite[Theorem 8.1]{BLP19}. Note that we add an additional condition that $\{ K_i(p) \}$ are linearly disjoint over $k$. This is because that the proof relies on \cite[Proposition~5.13]{BLP19}, which should be modified by adding this condition; see also Remark~\ref{rem:2.6}. \qed
\end{proof}

By Theorem~\ref{thm:BLP.2}, it is important to compute $\Sha(L)$ in the case where $L$ is a product of cyclic extensions of degree $p$. More generally we have the following result 
\cite[Proposition 8.5]{BLP19}.
\begin{prop}\label{prop:BLP.3}
Let $p$ be a prime number, and $L=\prod_{i=0}^m K_i$ a product of distinct field extensions of degree $p$ such that $K_0/k$ is cyclic. 
Then $\Sha(L)\neq 0$ only if every field $K_i$ is contained a field extension $F/k$ of degree $p^2$ and all local degrees of $F$ are $\le p$. 
Moreover, if the above condition is satisfied, then $\Sha(L)\simeq (\Z/p\Z)^{m-1}$.
\end{prop}



\begin{remark}\label{rem:2.6}
Theorem 8.1 of \cite{BLP19} is incorrect as stated. 
Lee gave 
a counterexample to it (\cite[Example 7.7]{lee2022tate}):
Let $k=\Q(i), K_0 = k ( \sqrt[4]{13}), 
K_1 = k(\sqrt[4]{17})$ and 
$k(\sqrt[4]{13\cdot 17^2})$. 
We have $\Sha(L)=\Z/2\Z$, while $L_{\rm prim}=k ( \sqrt{13})\times k(\sqrt{17})\times k ( \sqrt{13})$ and $\Sha(L_{\rm prim})=0$. 
\end{remark}

\subsection{}
In what follows, we let $L=\prod_{i=0}^{m} K_i$, where $K_i$ are \emph{cyclic extensions} of $k$ of degree $p^{\epsilon_{i}}$. Assume ${\cap}_{i=0}^{m} K_{i}=k$ and $\epsilon_0=\min_{0\le i\le m}\{\epsilon_i\}$. 
For any $i,j\in\mathcal{I}$, we set
\begin{enumerate}
    \item[(i)]
    $p^{e_{i,j}}=[K_{i}\cap K_{j}:k]$, and
    \item[(ii)]
    $e_{i}=\epsilon_{0}-e_{0,i}$.
\end{enumerate}
Without loss of generality, we assume that $e_{i}\geq e_{i+1}$, and notice that $e_{1}=\epsilon_{0}$ since $K_0\cap K_1=k$. 
Note that $p^{e_i}=[M_i:K_i]$, where $M_i=K_0K_i$ and one has $H^1(k, \wh T_{E_i/K_i})\simeq \Z/p^{e_i} \Z$. 

For any $0\le d \le \epsilon_i$, let $K_i(d)$ denote the subfield of $K_i$ of degree $p^d$ over $k$.\footnote{Do not confuse this with the notation $K_i(p)$ in Section~\ref{sec:S.1}.} For a nonempty subset $c\subseteq\mathcal{I}$ and an integer $d>0$, set $M_{c}(d):=\langle K_{i}(d)\rangle_{i\in c}$, the composite of $K_i(d)$ for $i\in c$. For $0\leq r\leq\epsilon_{0}$. Set
\[ \text{$U_{r}:=\{i\in\mathcal{I}|\text{ }e_{0,i}=r\}$, \quad $U_{>r}:=\{i\in\mathcal{I}|\text{ }e_{0,i}>r\}$,\quad and \quad $U_{<r}:=\{i\in\mathcal{I}|\text{ }e_{0,i}<r\}$.} \]

In order to describe formulas for the Tate-Shafarevich group of multinorm one tori due to T.-Y. Lee, we need to introduce the following invariants.
\begin{defn}
For a nonempty set $U_{r}$, the \emph{algebraic patching degree} $\Delta_{r}^{\omega}$ of $U_{r}$ is the largest nonnegative integer $d\le \epsilon_0$ satisfying the following two conditions:
\begin{enumerate}[label = (\roman*)]
    \item If $U_{>r}$ is nonempty, then $M_{U_{>r}}(d)\subseteq\mathop{\cap}\limits_{i\in{U_{r}}}K_{0}(d)K_{i}(d)$.
    \item If $U_{<r}$ is nonempty, then $M_{U_{r}}(d)\subseteq\mathop{\cap}\limits_{i\in U_{<r}}K_{0}(d)K_{i}(d)$.
\end{enumerate}
If $U_{r}=\mathcal{I}$ (so $r=0$), then we set $\Delta_{0}^{\omega}=\epsilon_{0}$.
\end{defn}
    
We say that a field extension $M$ of $k$ is \emph{locally cyclic} if its completion $M\otimes_{k} k_{v}$ at $v$ is a product of cyclic extensions of $k_v$ for all places $v\in\Omega_k$.  Moreover, if $M$ is a finite Galois extension of $k$, then $M/k$ is locally cyclic if and only if every decomposition group of $M$ over $k$ is cyclic.
\begin{defn}
The \emph{patching degree} $\Delta_{r}$ of $U_{r}$ is the largest nonnegative integer $d\le \Delta^\omega_r$ satisfying the following two conditions:
\begin{enumerate}
    \item[\rm{(i)}] 
    If $U_{>r}$ is nonempty, then
    $K_{0}(d)M_{U_{>r}}(d)$ is locally cyclic. \item[\rm{(ii)}]
    If $U_{<r}$ is nonempty, then $K_{0}(d)M_{U_{r}}(d)$ is locally cyclic. 
\end{enumerate}
If $U_{0}=\mathcal{I}$, then we set $\Delta_{0}=\epsilon_{0}$.
\end{defn}
\begin{defn}
Let $i,j \in \mathcal{I}$ and $l$ be a nonnegative integer. We say that $i,j$ are \emph{$l$-equivalent} and denoted by $i\mathop{\sim}\limits_{l}j$ if $e_{i,j}\geq l$ or $i=j$. For any nonempty subset $c$ of $\mathcal{I}$, let $n_{l}(c)$ be the number of $l$-equivalence classes of $c$.
\end{defn}
\begin{defn}
For each $c \subseteq\mathcal{I}$ with $|c|\geq 1$, the \emph{level} of $c$ is defined by 
\[L(c):=\min \{e_{i,j}: i,j\in c\}. \]  
\end{defn}
\begin{defn}
\begin{enumerate}[label = (\arabic*)]
    \item For a nonempty set $U_{r}$, let $l_{r}=L(U_{r})$ and let $f^{\omega}_{U_r}$ be the largest nonnegative integer $f \leq \Delta_{r}^{\omega}$ satisfying the following two conditions:
    \begin{enumerate}[label = (\roman*)]
        \item The field $M_{U_{r}}(f+l_{r}-r)$ is a subfield of a bicyclic extension.
        \item $K_{0}(f)\subseteq M_{U_{r}}(f+l_{r}-r)$.
    \end{enumerate}
    We call $f_{U_{r}}^{\omega}$ the \emph{algebraic degree of freedom} of $U_{r}$.

    \item Similarly, for any $h$-equivalence class $c\subset U_r$ with $h\geq L(U_r)$, the \emph{algebraic degree of freedom} of $c$, denoted by $f_{c}^{\omega}$, is the largest nonnegative integer $f\leq\Delta_{r}^{\omega}$ satisfying the following two conditions:
    \begin{enumerate}[label = (\roman*)]
        \item The field $M_{c}(f+L(c)-r)$ is a subfield of a bicyclic extension.
        \item $K_0(f)\subseteq M_{c}(f+L(c)-r)$.
    \end{enumerate}
\end{enumerate}
  
\end{defn}
 

According to the definition one has $r\le f_c^\omega \le f_{U_r}^\omega \le \Delta^\omega_r$.   
  
\begin{defn}
Let $c\subset U_r$ be an $h$-equivalence class for some $h\geq L(U_{r})$. The \emph{degree of freedom} $f_{c}$ of $c$ is defined to be the largest nonnegative integer $f\leq f_{c}^{\omega}$ such that $M_{c}(f+L(c)-r)$ is locally cyclic.
\end{defn}

\begin{thm} \cite[Theorem 6.5]{lee2022tate}  \label{thm:Lee}
Let $T_{L/k}$ be the multinorm one torus associated to a $k$-\'etale algebra $L=\prod_{i=0}^m K_i$. We have
$$
\Sha_{\omega}^{2}(k,\widehat{T}_{L/k})\cong\mathop{\oplus}\limits_{r\in\mathcal{R}\setminus\{0\}}\mathbb{Z}/p^{{\Delta_{r}^{w}}-r}\mathbb{Z}\mathop{\oplus}\limits_{r\in\mathcal{R}}\mathop{\oplus}\limits_{l\geq L(U_{r})}\mathop{\oplus}\limits_{c\in U_r/\mathop{\sim}\limits_{l}}(\mathbb{Z}/p^{f_{c}^{w}-r}\mathbb{Z})^{n_{l+1}(c)-1};$$ 
$$\Sha^{2}(k,\widehat{T}_{L/k})\cong\mathop{\oplus}\limits_{r\in\mathcal{R}\setminus\{0\}}\mathbb{Z}/p^{{\Delta_{r}}-r}\mathbb{Z}\mathop{\oplus}\limits_{r\in\mathcal{R}}\mathop{\oplus}\limits_{l\geq L(U_{r})}\mathop{\oplus}\limits_{c\in U_r/\mathop{\sim}\limits_{l}}(\mathbb{Z}/p^{f_{c}-r}\mathbb{Z})^{n_{l+1}(c)-1},
$$
where $\mathcal{R}=\{0\leq r\leq\epsilon_0|\text{ } U_{r}\neq\emptyset\}$.
\end{thm}

\subsection{} We use Theorem~\ref{thm:Lee} to revisit the criterion for the vanishing of the groups $\Sha_{\omega}^{2}(k,\widehat{T}_{L/k})$ and $\Sha^{2}(k,\widehat{T}_{L/k})$ \cite[Theorem 8.1]{BLP19}. 

\begin{defn}
A subset $c \subset \calI$ with $|c|>1$ is said to be \emph{admissible} if $c$ is an $l$-equivalence class in $U_r$ for some $r\ge 0$. The integer $r$, denoted ${\rm supp}(c)$, is called the support of $c$. Let ${\rm Adm}$ be the set of admissible subsets of $\calI$.
\end{defn}

Theorem~\ref{thm:Lee} can be reformulated as follows.

\begin{thm}\label{thm:reformulate} We have 
\begin{align*}
    \Sha_{\omega}^{2}(k,\widehat{T}_{L/k})&\cong\mathop{\oplus}\limits_{r\in\mathcal{R}\setminus\{0\}}\mathbb{Z}/p^{{\Delta_{r}^{w}}-r}\mathbb{Z}
    \oplus \mathop{\oplus}\limits_{c\in {\rm Adm}}
    (\mathbb{Z}/p^{f_{c}^{w}-r}\mathbb{Z})^{n_{L(c)+1}(c)-1}; \\
    \Sha^{2}(k,\widehat{T}_{L/k})&\cong\mathop{\oplus}\limits_{r\in\mathcal{R}\setminus\{0\}}\mathbb{Z}/p^{{\Delta_{r}}-r}\mathbb{Z} \mathop{\oplus} \mathop{\oplus}\limits_{c\in {\rm Adm}}
    (\mathbb{Z}/p^{f_{c}-r}
    \mathbb{Z})^{n_{L(c)+1}(c)-1}.
\end{align*}
\end{thm}

\begin{prop}\label{prop:vanishing}
Let $r_0>0$ be the smallest integer such that $U_r$ is nonempty.
\begin{enumerate} 
    \item[{\rm (1)}] We have $\Sha_{\omega}^{2}(k,\widehat{T}_{L/k})=0$ if and only if $\Delta^\omega_{r_0}=r_0$ and $f^\omega_{U_0}=0$.
    \item[{\rm (2)}] We have $\Sha^{2}(k,\widehat{T}_{L/k})=0$ if and only if $\Delta_{r_0}=r_0$ and $f_{U_0}=0$.
\end{enumerate}
\end{prop}
\begin{proof}
By Theorem~\ref{thm:reformulate}, $\Sha_{\omega}^{2}(k,\widehat{T}_{L/k})=0$ if an only if $\Delta^\omega_r=r$ for all $r\in \calR\setminus\{0\}$ and $f^\omega_c=r$ for all admissible subsets $c$ of support $r$. Since $r\le f^\omega_c \le \Delta^\omega_r$, the first condition $\Delta^\omega_r=r$ implies that $f^\omega_c=r$ for all admissible subsets $c$ of support $r\ge 1$. Also one has $0\le {f^\omega_c} \le f^\omega_{U_0}$ if $c\subset U_0$, so that the above condition is equivalent to that $\Delta^\omega_r=r$ for all $r\in \calR\setminus\{0\}$ and $f^\omega_{U_0}=0$. By \cite[Proposition 4.3]{lee2022tate}, we have 
$\Delta^\omega_{r_0}-r_0 \ge \Delta^\omega_{r}-r$. This proves the first statement. 

We now show $r\le f_c \le f_{U_r} \le \Delta_r$. By \cite[Proposition 5.8]{lee2022tate}, if $r\le f \le f^\omega_c$ and $i\in c$, then
\begin{align*}
    M_c(f+L(c)-r)=K_0(f) K_i(f+L(c)-r).
\end{align*} 
Since $f\le f^\omega_{U_r}$, one also has
\begin{align*}
    M_{U_r}(f+L(U_r)-r)=K_0(f) K_i(f+L(U_r)-r).
\end{align*}
Therefore, $M_{U_r}(f+L(U_r)-r)\subset M_c(f+L(c)-r)$. Hence if $M_c(f+L(c)-r)$ is locally cyclic then $M_{U_r}(f+L(U_r)-r)$ is locally cyclic. It follows that $f_c\le f_{U_r}$. 

For $r\le f\le f^\omega_{U_r} \le \Delta^\omega_r$ and $i\in U_r$, one has 
\[ K_0(f)K_i(f) \subset K_0(f) K_i(f+L(U_r)-r)=M_{U_r}(f+L(U_r)-r)\]
and hence $K_0(f)K_{U_r}(f)\subset M_{U_r}(f+L(U_r)-r)$. Since $f\le \Delta^\omega_r$, one also has
\[ K_0(f)K_{U_{>r}}(f) \subset \bigcap_{i\in U_r} K_0(f)K_i(f)\subset K_0(f)K_i(f) \subset M_{U_r}(f+L(U_r)-r). \] 
Therefore, if $M_{U_r}(f+L(U_r)-r)$ is locally cyclic, then both $K_0(f)K_{U_{>r}}(f)$ and $K_0(f)K_{U_r}(f)$ are locally cyclic. It follows that $f_{U_r}\le \Delta_r$. This shows $r\le f_c \le f_{U_r} \le \Delta_r$. 

The second statement then follows from the same argument and \cite[Proposition 4.10]{lee2022tate}. \qed
\end{proof}

We can look further the conditions for nonvanishing of the groups $\Sha_{\omega}^{2}(k,\widehat{T}_{L/k})$ and $\Sha^{2}(k,\widehat{T}_{L/k})$.
Note that if $U_{>r}\neq \emptyset$, then $M_{U_{>r}}(r+1)=K_0(r+1)$ and hence the condition
\begin{align*}
    M_{U_{>r}}(r+1) \subset \bigcap_{i\in U_r} K_0(r+1) K_i(r+1)
\end{align*}
always holds. Thus, we have
\begin{equation}
    \Delta^\omega_{r_0}\ge r_0+1 \iff  M_{U_{r_0}}(r_0+1)\subset \bigcap_{i\in U_{< r_0}} K_0(r_0+1) K_i(r_0+1).
\end{equation}
By definition, we have $\Delta_{r_0}\ge r_0+1$ if and only if the
following three conditions hold 
\begin{itemize}
    \item [(i)] $\Delta^\omega_{r_0}\ge r_0+1$;
    \item[(ii)] $K_{0}(r_0+1)M_{U_{r_0}}(r_0+1)$ is locally cyclic;
    \item[(iii)] If $U_{>r_0}$ is nonempty, then $K_{0}(r_0+1)M_{U_{>r_0}}(r_0+1)$ is locally cyclic.
\end{itemize}
Similarly, we have 
\begin{equation} \label{eq:2.6}
   f^\omega_{U_0}\ge 1 \iff
   \begin{array}{l}
        \text{$M_{U_0}(1+L(U_0))$ is a subfield of a bicyclic extension of $k$} \\
        \text{and it contains $K_0(1)$,}
   \end{array}
\end{equation}
and 
\begin{equation}\label{eq:2.7}
   f_{U_0}\ge 1 \iff \text{$f^\omega_{U_0}\ge 1$ and $M_{U_0}(1+L(U_0))$ is locally cyclic.}
\end{equation}

In the special case where $K_i$ are distinct cyclic extensions of degree $p$ over $k$, one obtains from \eqref{eq:2.7} that $\Sha(L)\neq 0$ if and only if (i) $K_0\subset M_{U_0}(1)$, (ii) $M_{U_0}(1)$ is a subfield of a bicyclic extension, and (iii) $M_{U_0}(1)$ is locally cyclic. This is the same as Proposition~\ref{prop:BLP.3} in this special case.

\section{Remarks on the conditions for Theorem~\ref{thm:Lee}}
\label{sec:R}

\subsection{}
Note that the assumption $e_1 \geq e_2 \geq \dotsb \geq e_m$ is unnecessary. We can choose some permutation $\sigma \in S_m$ such that $e_{\sigma(i)} \geq e_{\sigma(i+1)}$. The invariants $e_{i, j}$ are identical up to $\sigma$. From the definition of $\ell$-equivalence, $U_r$, $\epsilon_i$, (algebraic) patching degrees $\Delta^{(\omega)}_r$, (algebraic) degrees of freedom $f^{\omega}_c$, etc., we see that they are identical after the action of $\sigma$. Therefore the Tate-Shafarevich groups $\Sha(L)$ and $\Sha_{\omega}(L)$ given by the formula without the assumption are the same as those given by the formula with the assumption. Thus, for implementing an algorithm, we do not need to rearrange of our input data so that this assumption holds. 

\subsection{}
In this subsection, we discuss whether we have the same results without the condition $\cap_{i = 0}^m K_i = k$. That is, setting $F = \cap_{i = 0}^m K_i$ and considering $L/F$ as an étale $F$-algebra, we compare the groups $\Sha^2(k, \hat{T}_{L/k})$ and $\Sha^2(F, \hat{T}_{L/F})$.

First we denote 
\[ T^L_F = R_{L/F} \mathbb{G}_{m, L}, \quad T^L := R_{L/k} \mathbb{G}_{m, L}=R_{F/k} T^L_F, \quad T^F = R_{F/k}\mathbb{G}_{m, F}, \]
and let $T_{L/F} = \Ker N_{L/F}$ and $T_{L/k} = \Ker N_{L/k}$, where $N_{L/F} = \prod_{i = 0}^m N_{K_i/F}$ and $N_{L/k} = \prod_{i = 0}^m N_{K_i/k}$ are the norm maps. Let $\Tilde{k}=K_0 K_1 \dotsb K_m$ be the composition of $K_i$, and set 
\[ G = \Gal(\Tilde{k}/k), \quad H_i = \Gal(\Tilde{k}/K_i), \quad \text{and} \quad H = \Gal(\Tilde{k}/F). \]

\begin{lemma}\label{lm:3.1}
\begin{enumerate}
    \item We have $H^1(F,\wh T_{L/F})=H^1(H,\wh T_{L/F})=0$.
    \item We have $H^1(F_w,\wh T_{L^w/F_w})=H^1(H_w,\wh T_{L/F})=0$, where $w$ is a place of $F$, $H_w$ is the decomposition group, and $L^w=L\otimes F_w$.  
\end{enumerate}
   
\end{lemma}
\begin{proof}
\begin{enumerate}
    \item Since the group $H^1(H,\wh T_{L/F})$ is independent of the choice of the splitting field, we have $H^1(H,\wh T_{L/F})=H^1(F,\wh T_{L/F})$. From the exact sequence
    $ 1\rightarrow T_{L/F} \rightarrow T^{L}_F \rightarrow\mathbb{G}_{m,F} \rightarrow 1$, we obtain the exact sequence  
 $$ 0\longrightarrow \mathbb{Z}\longrightarrow \widehat{T^L_F}=\bigoplus_{i=0}^m \Ind_{H_i}^H \Z 
\longrightarrow\widehat T_{L/F}\longrightarrow 0$$   
and hence the long exact sequence 
\begin{equation}\label{eq:H1T_hat}
H^1(H,\widehat{T^L_F})\longrightarrow H^1(H,\widehat T_{L/F})\longrightarrow H^2(H,\mathbb{Z})\longrightarrow H^2(H,\widehat{T^L}).
\end{equation}
It follows from Shapiro's Lemma that $H^1(H,\widehat{T^L})=\oplus H^1(H_i, \Z)=0$. Using the canonical isomorphism $H^2(H,\Z)\simeq H^1(H,\Q/\Z)=\Hom(H,\Q/\Z)$
we get from \eqref{eq:H1T_hat} that 
$$H^1(H,\widehat T_{L/F})\simeq \Ker\left (\Hom(H,\mathbb{Q/Z})\to  \bigoplus_{i=0}^m \Hom(H_i,\mathbb{Q/Z})\right ). $$
Since $\cap_i K_i=F$, one has $H=H_0\cdots H_m$ and hence  $H^1(H,\widehat T_{L/F})=0$.
\item By the construction, we have $T_{L/F}\otimes_F F_v=T_{L^w/F_v}$. Thus, the $\Gamma_{F_v}$-module $\wh T_{L^w/F_w}$ is equal to $\wh T_{L/F}$ when viewed as a $\Gamma_{F_v}$ by an inclusion $\Gamma_{F_v}\embed \Gamma_{F}$. As its first Galois cohomology is independent of the choice of a splitting field, one gets
\[ H^1(F_w, \wh T_{L^w/F_w})\simeq H^1(H_w, \wh T_{L^w/F_w})\simeq H^1(H_w, \wh T_{L/F}).\]
By the same argument as (1), one obtains 
$H^1(F_w, \wh T_{L^w/F_w})=0$. This completes the proof of the lemma. \qed
\end{enumerate}
 
\end{proof}

\begin{lemma}\label{lm:Sha2hatT}
   Let $T$ be an algebraic torus over $k$ and $K/k$ a Galois splitting field for $T$ with Galois group $G$.
   \begin{enumerate}
       \item There is a natural isomorphism $\Sha^2(G, \wh T)\isoto \Sha^2(k, \wh T)$. 
       \item There is a natural isomorphism $\Sha^2_{\omega}(G, \wh T)\isoto \Sha^2_{\omega}(k, \wh T)$.
   \end{enumerate}
\end{lemma}
\begin{proof}
These are well-known results. The statement (1) follows from the fact that the group $\Sha^1(G, T)$ is independent of the choice of the splitting field $K$;  see \cite[Sections 3.3 and 3.4]{ono:tamagawa} and the Poitou-Tate duality (\cite[Theorem 6.10]{platonov-rapinchuk} and \cite[Theorem 8.6.8]{NSW}). We give a proof of (2) for the reader's convenience. Let $K'$ be another Galois splitting field for $T$ containing $K$ with Galois groups $G'=\Gal(K'/k)$ and $H=\Gal(K'/K)$. Since $\wh T$ is a trivial $H$-module, $H^1(H, \wh T)=\Hom(H, \wh T)=0$. By  
Hochschild-Serre's spectral sequence, we have the exact sequence 
\[ 0\longrightarrow H^2(G, \wh T) \longrightarrow H^2(G', \wh T) \longrightarrow H^2(H, \wh T). \]
Thus, to show $\Sha^2_\omega(G, \wh T) \isoto \Sha^2_\omega(G', \wh T)$, it suffices to show $\Sha^2_\omega(H, \wh T)=0$. Since $\wh T$ is a trivial $H$-module, it is the same to show $\Sha^2_\omega(H, \Z)=0$. As $H^2(H,\Z)\simeq H^1(H, \Q/\Z)$, this follows from that 
\[ \Ker \left( \Hom(H, \Q/\Z) \longrightarrow \prod_{C}\Hom(C,\Q/\Z) \right )=0, \] 
where $C$ runs through all cyclic subgroups of $H$. \qed
\end{proof}

\begin{prop} \label{prop:TLk_TLF}
   There is a natural injective map $\wh \iota: \Sha^2_\omega(k, \hat{T}_{L/k})  \longrightarrow  \Sha^2_\omega(F, \hat{T}_{L/F})$. 
\end{prop}

\begin{proof}
 From the commutative diagram
\begin{center}
    \begin{tikzcd}
        1 \arrow[r] & R_{F/k}(T_{L/F}) \arrow[r] \arrow[d, hook] & T^L \arrow[r, "N_{L/F}"] \arrow[d, equal] & T^F \arrow[r] \arrow[d, "N_{F/k}"] & 1 \\
        1 \arrow[r] & T_{L/k} \arrow[r] & T^L \arrow[r, "N_{L/k}"] & \Gmk \arrow[r] & 1
    \end{tikzcd}
\end{center}
where the two rows are exact, we obtain the following commutative diagram
\begin{center}
    \begin{tikzcd}
        1 \arrow[r] & R_{F/k} T_{L/F} \arrow[r, "\iota"] \arrow[d, equal] & T_{L/k} \arrow[r, "N_{L/F}"] \arrow[d, hook] & T_{F/k} \arrow[r] \arrow[d, hook] & 1 \\
        1 \arrow[r] & R_{F/k} T_{L/F} \arrow[r, "\iota"] & T^L \arrow[r, "N_{L/F}"] \arrow[d, "N_{L/k}"] & T^F \arrow[r] \arrow[d, "N_{F/k}"] & 1 \\
        & & \Gmk \arrow[r, equal] \arrow[d] & \Gmk \arrow[d] & \\
        & & 1 & 1 &
    \end{tikzcd}
\end{center}
whose rows and columns are exact. Taking the dual of the first row  yields an exact sequence 
\begin{equation}\label{eq:3.1}
    \begin{tikzcd}
       0 \arrow[r] & \wh T_{F/k} \arrow[r, "\wh N_{L/F}" ]  & \wh T_{L/k}  \arrow[r, "\wh \iota"]  & \Ind_H^G \wh T_{L/F}  \arrow[r] &  0.
    \end{tikzcd}
\end{equation}
This gives the following commutative diagram
\begin{equation}\label{eq:3.2}
    \begin{tikzcd}
       H^1(H, \wh T_{L/F}) \arrow[r] \arrow[d] &  H^2(G,\wh T_{F/k}) \arrow[r, "\wh N_{L/F}" ] \arrow[d, "r_{F/k}"] & H^2(G, \wh T_{L/k})  \arrow[r, "\wh \iota"] \arrow[d, "r_{L/k}"] & H^2(H, \wh T_{L/F}) \arrow[d, "r_{L/F}"] \\
       \prod_{w|v} H^1(H_w, \wh T_{L/F}) \arrow[r] & H^2(G_v,\wh T_{F/k}) \arrow[r, "\wh N_{L/F,v}" ]  & H^2(G_v, \wh T_{L/k})  \arrow[r, "\wh \iota_v"]  & \prod_{w|v} H^2(H_w, \wh T_{L/F})      
    \end{tikzcd}
\end{equation}
for every decomposition group $G_v$ of $G$, where $w$ runs through places of $F$ over $v$. By Lemma~\ref{lm:3.1}, we have $H^1(H,\wh T_{L/F})=0$ and $H^1(H_w,\wh T_{L/F})=0$, 
so the maps $\wh N_{L/F}$ and $\wh N_{L/F,v}$ are injective. 
Suppose an element $x\in H^2(G,\wh T_{L/k})$ lies in $\Ker \wh \iota$ satisfying $r_{L/k}(x)=0$. Let $y\in H^2(G,\wh T_{F/k})$ be the unqiue element with $\wh N_{L/F}(y)=x$. Then $r_{F/k}(y)=0$ as the map $\wh N_{L/F,v}$ is injective. It follows that $\Sha_\omega^2(G,\wh T_{L/k})\cap \Ker \wh \iota\simeq \Sha^2_\omega (G,\wh T_{F/k})$, which is zero from \cite[Proposition 2.2]{lee2022tate} as $F/k$ is cyclic.
Thus, the map
\begin{align*}
    \wh \iota: \Sha^2_\omega(k, \wh{T}_{L/k}) = \Sha^2_\omega(G, \wh{T}_{L/k}) \hookrightarrow \Sha^2_\omega(H, \wh{T}_{L/F}) = \Sha^2_\omega(F, \wh{T}_{L/F})
\end{align*}
is injective. \qed   
\end{proof}

\begin{cor}
    Notation being as above, if $\Sha^2_{(\omega)}(F, \wh{T}_{L/F})=0$, then $\Sha^2_{(\omega)}(F, \wh{T}_{L/k})=0$.
\end{cor}

\section{Multinorm one tori of Kummer type}\label{sec:K}

\subsection{Kummer extensions}
For a moment, let $k$ be a field which contains a primitive $N$-th root of unity, where $N \ge 2$ is a positive integer prime to the characteristic of $k$. Recall that a Kummer extension $L/k$ of exponent $N$ is a finite abelian field extension, whose Galois group $\Gal(L/k)$ is of exponent $N$, that is, $\sigma^N = 1$ for any $\sigma \in \Gal(L/k)$. For example, if $\char k \neq 2$ then a quadratic extension $L = k(\sqrt{a})$, where $a \in k$ is not a square, is a Kummer extension. Biquadratic extensions and multiquadratic extensions are also Kummer extensions. More generally, for any nonzero element $a \in k$, $k(a^{1/N})$ is a Kummer extension whose degree $m$ divides $N$.

Kummer theory establishes the following one-to-one correspondence
\begin{align} \label{Kummer Theory}
    \left\{ \,\text{Kummer extensions over $k$ of exponent $N$} \, \right\} \longleftrightarrow \left\{ \,\text{finite subgroups of}\, k^{\times}/(k^{\times})^N \, \right\}.
\end{align}
For any finite subgroup $W$ of $k^{\times}/(k^{\times})^N$, we define
\begin{align*}
    K_W \coloneqq k\left( w^{1/N}: w \in W \right)
\end{align*}
and associate $K_W$ to $W$. Conversely, let $L$ be a Kummer extension of $k$. Since $L$ is of exponent $N$, $L$ can be written as a composition of cyclic extensions $k(a_1^{1/N}) \dotsb k(a_m^{1/N})$, where $a_i \in k^{\times}$. We associate it to the subgroup
\begin{align*}
    W_L = \langle \overline{a_i} \mid i = 1, \dotsc, m \rangle
\end{align*}
where $\overline{a_i}$ denotes the image of $a_i$ in $k^{\times}$. 

Let $\mu_N$ be the group of $N$-th roots of unity in $k^{\times}$. There is a perfect pairing 
\begin{align*}
    \Gal(K_W/k) \times W &\longrightarrow \mu_N \\
    (\sigma,w) &\longmapsto \frac{\sigma(w^{1/p^n})}{w^{1/p^n}}.
\end{align*}
This gives a natural identification $\Gal(K_W/k)=\Hom(W,\mu_N)$. If $W_1 \subset W_2$ are two subgroups of $k^\times/(k^\times)^{p^n}$,
the natural projection $\Gal(K_{W_2}/k) \to \Gal(K_{W_1}/k)$ is the restriction to $W_1$:
\begin{align*}
    \Hom(W_2, \mu_N) \longrightarrow \Hom(W_1, \mu_N).
\end{align*}

Inclusion, composition, and intersection of groups $W_i$ correspond to those of Kummer extensions.

\begin{prop}
Let $W$ and $W_i$ ($i=1,2$) be subgroups of $k^{\times}/(k^{\times})^N$ and $K_{W}$ and  $K_{W_i}$ be the corresponding Kummer extensions. Then
\begin{enumerate}
    \item $K_{W_1} \subset K_{W_2}$ if and only if $W_1 \subset W_2$.
    \item $W = W_1 W_2$ if and only if $K_W = K_{W_1} K_{W_2}$.
    \item $W = W_1 \cap W_2$ if and only if $K_{W_1} \cap K_{W_2} = K_W$.
\end{enumerate}
\end{prop}

\subsection{Group theoretic description for $\Sha^2_\omega(k, \wh T_{L/k})$ and $\Sha^2(k, \wh T_{L/k})$}

For the rest of this section, let $k$ be a global field in which $p^{-1}\in k$ and $L=\prod_{i=0}^m K_i$ an étale $k$-algebra as in Section 2.2. 
Let $N = p^n$ be a power of $p$ such that $[K_i:k]$ divides $N$ for all $i$. Suppose that $k$ contains a primitive $N$-th root of unity. We further assume that each $K_i$ can be written as $k(\alpha_i)$, where $\alpha_i = {a_i}^{1/p^n}$ for some $a_i \in \Q^{\times}$. We may assume $a_i \in \Z$: if $a_i = \frac{a}{b}$, we can set  $a'_i = a_i b^{p^n}$ so that $k({a_i}^{1/p^n}) = k({a'_i}^{1/p^n})$.

The correspondence \eqref{Kummer Theory} enables us to describe $\Sha^2_\omega(k, \wh T_{L/k})$ and $\Sha^2(k, \wh T_{L/k})$ in terms of information in the group $k^{\times}/(k^{\times})^{p^n}$. First, we set $W_i = \langle \overline{a_i} \rangle$ to be the subgroup corresponding to $K_i$. For any nonempty subset $I$ of $\mathcal{I} = \{1, \dotsc, m\}$, we let $W_I = \langle \overline{a_i} \mid i \in I \rangle$ be the group corresponding to $M_I$. We define $W_i(d)$, $W_I(d)$ as groups corresponding to $K_i(d)$ and $M_I(d)$, respectively. Note that the order of $\overline{a_i}$ in $k^{\times}/(k^{\times})^{p^n}$ is $p^{\epsilon_i}$, so $K_i(d) = k(a_i^{p^{\epsilon_i - d}/p^n})$ and $W_i(d) = \left\langle \overline{a_i^{p^{\epsilon_i - d}/p^n}} \right\rangle$.

We translate the first definitions in Section 2 as follows.
\begin{enumerate}
    \item For $i, j \in I$, $i$ and $j$ are $\ell$-equivalent if and only if $W_i(\ell) = W_j(\ell)$.
    \item The set $U_r = \{ i \in \mathcal{I} \mid W_0(r) = W_0 \cap W_i = W_i(r) \}$.
    \item For any subset $c \subset \mathcal{I}'$, $L(c) = \min\left\{ \ell \mid W_i(\ell) = W_i \cap W_j = W_j(\ell) \,\text{ for any }\, i, j \in c \right\}$.
\end{enumerate}

With the above language, we can rewrite the definitions of algebraic patching degrees and algebraic degrees of freedom. If $U_0 = \mathcal{I}$, then we set the algebraic patching degree $\Delta^{\omega}_0 = \epsilon_0$. Otherwise, the algebraic patching degree of freedom $\Delta^{\omega}_r$ for nonempty $U_r$ is the maximal positive integer $d$ satisfying two conditions:
\begin{enumerate}[label = (\roman*)]
    \item If $U_{>r}$ is nonempty, then $W_{U_{>r}}(d) \subset \bigcap\limits_{i \in U_r} W_0(d) W_i(d)$.
    \item If $U_{<r}$ is nonemtpy, then $W_{U_r}(d) \subset \bigcap\limits_{i \in U_{< r}} W_0(d) W_i(d)$.
\end{enumerate}
Now the algebraic degree of freedom $f_c^{\omega}$ for an admissible set $c \subset U_r$ can be defined as the largest nonnegative integer $f\le \Delta^{\omega}_r$ satisfying two conditions:
\begin{enumerate}[label = (\roman*)]
    \item $W_c(f + L(c) - r)$ is a cyclic group or a bicyclic group.
    \item $W_0(f) \subset W_c(f + L(c) - r)$.
\end{enumerate}

Before we rewrite the definition of patching degrees and degrees of freedom, recall that we have to check whether a field is locally cyclic in the definition of patching degrees $\Delta_r$. We need to describe whether a Kummer extension is locally cyclic in terms of groups, too. Let $K/k$ be a Kummer extension and $v$ a place of $k$. Let $w$ be a place of $K$ lying over $v$. The decomposition group $G_v = \Gal(K_w/k_v)$ corresponds to a subgroup $W_v$ of $k_v^{\times}/(k_v^{\times})^{p^n}$ through the duality between $\Gal(K/k)$ and $W$.
\begin{center}
    \begin{tikzcd}
        \Gal(K_w/k_v) = G_v \arrow[d, hookrightarrow] \arrow[r, leftrightarrow] & W_v \subset k_v^{\times}/(k_v^{\times})^{p^n} \\
        \Gal(K/k) \arrow[r, leftrightarrow] & W \subset k^{\times}/(k^{\times})^{p^n} \arrow[u, twoheadrightarrow, "\pi_v"']
    \end{tikzcd}
\end{center}
The map $\pi_v: W \to W_v$ is induced by the map $k^{\times} \hookrightarrow k_v^{\times}$ whose image is dense in $k_v^{\times}$. $W_v$ is a finite set and hence $\pi_v$ is surjective. Recall that $K/k$ is locally cyclic at $v$ means that $K_w/k_v$ is cyclic for any $w \mid v$, and this is equivalent to saying that $\pi_v(W)$ is cyclic for any $v$. 

Now we can redefine the patching degree $\Delta_r$ to be the maximal positive integer $d \leq \Delta^{\omega}_r$ satisfying two conditions:
\begin{enumerate}[label = (\roman*)]
    \item If $U_{> r}$ is nonempty then $\pi_v(W_0(d) W_{U_{>r}}(d))$ is cyclic for all places $v$ in $k$.
    \item If $U_{< r}$ is nonempty then$\pi_v(W_0(d) W_{U_r}(d))$ is cyclic for all places $v$ in $k$.
\end{enumerate}
On the other hand, for an admissible set $c \subset U_r$ the degree of freedom $f_c$ is the largest nonnegative integer $f\le f_c^\omega$ such that $\pi_v(W_c(f + L(c) - r))$ is a cyclic group for any place $v$ of $k$.

\subsection{Cyclotomic cases: combinatorial description for $\Sha^2_\omega(k,\wh T_{L/k})$ and $\Sha^2(k, \wh T_{L/k})$}

In this subsection we define
\begin{align*}
    \mathcal{W} = \langle \overline{a}: a \in \Q^{\times} \rangle \subset k^{\times}/(k^{\times})^{p^n},
\end{align*}
that is, the image of $\iota: \Q^{\times}/(\Q^{\times})^{p^n} \longrightarrow k^{\times}/(k^{\times})^{p^n}$. Because each component $K_i$ of $L$ is of the form $k(a_i^{1/p^n})$ where $a_i$ is an integer, the group $W_i$ corresponding to $K_i$ is contained in $\mathcal{W}$. We shall investigate the structure of $\mathcal{W}$. Let $\bbP$ denote the set of prime numbers in $\Q$.


\begin{prop}
\begin{enumerate}[label = (\arabic*)]
    \item If $p$ is odd, then $\mathcal{W} \simeq \Q_{>0}/(\Q_{>0})^{p^n}\simeq \bigoplus\limits_{\ell\in \bbP} \Z/p^n\Z$.
    \item Suppose $p = 2$ and we denote $\mathbb{P}$ the set of prime integers.
    \begin{enumerate}[label = (\alph*)]
        \item If $N = 2$, then $\mathcal{W} \simeq \bigoplus\limits_{\ell \in \mathbb{P} \cup \{-1\}} \Z/2\Z$.
        \item If $N = 4$, then $\mathcal{W} \simeq \Z/2\Z \times \bigoplus\limits_{\ell \in \mathbb{P}} \Z/4\Z$.
        \item If $N \geq 8$, then $\mathcal{W} \simeq \Z/2\Z \times \Z/2^{n-1}\Z \times \bigoplus\limits_{\ell \in \mathbb{P} \setminus \{2\}} \Z/2^n\Z$.
    \end{enumerate}
\end{enumerate}
\end{prop}
\begin{proof}
\begin{enumerate}[label = (\arabic*)]
    \item Suppose $a$ is an integer such that $a = \alpha^{p^n}$ for some $\alpha \in k^{\times}$. Let $\ell$ be a prime integer not equal to $p$, then $\ell$ is unramified in $k$ and hence the valuation $v_{\ell}$ sends each element of $k^{\times}$ to an integer. Therefore, $v_{\ell}(a) = p^n v(\alpha) \in p^n \Z$, i.e., $p^n \mid n_i$ for all $i$. We may then assume that $a = \pm p^r$ for some $r \in \Z_{\geq 0}$. Because $\iota(-1) = 1$, we may further assume $a = p^r$. If $r > 0$, without loss of generality, we can write $p^p = \alpha^{p^n}$, i.e., $p = \alpha^{p^{n-1}}$ where $\alpha \in k = \Q(\zeta_{p^n})$. This shows that $\sqrt[p^{n-1}]{p} \in \Q(\zeta_{p^n})$. On the other hand, we consider the Galois group of $k(p^{1/p}) = \Q(\zeta_{p^n}, p^)$. In this group we have automorphisms
    \begin{align*}
        \tau_a : \zeta_{p^n} \mapsto \zeta_{p^n}^a, \quad \sqrt[p^{n-1}]{p} \mapsto \sqrt[p^{n-1}]{p}
    \end{align*}
    for $a \in (\Z/p^n \Z)^{\times}$. Also we have
    \begin{align*}
        \sigma: \zeta_{p^n} \mapsto \zeta_{p^n}, \quad \sqrt[p^{n-1}]{p} \mapsto \sqrt[p^{n-1}]{p} \, \zeta_{p^{n-1}}.
    \end{align*}
    But $\tau_a \sigma \tau_a^{-1} \neq \sigma$, so the Galois group is not abelian. This contradicts the above conclusion $\Q(\sqrt[p^{n-1}]{p}, \zeta_{p^n}) = \Q(\zeta_{p^n})$. Therefore, the integer $a$ must be $1$, and hence we conclude that $\ker(\iota) = \{\overline{1}\}$.
    
    \item When $N = p = 2$, $k$ is simply $\Q$ and thus
    \begin{align*}
        \mathcal{W} = \Q^{\times}/(\Q^{\times})^2 \simeq \bigoplus\limits_{\ell \in \mathbb{P} \cup \{-1\}} \Z/2\Z.
    \end{align*}

    Now suppose $N = 2^n \geq 4$. Observe that $\sqrt{2} \in \Q(\zeta_8)$ and $-1 \in (k^{\times})^2$. If $\ell$ is a prime integer other than $2$, then the argument in part (1) applies, so $\ker(\iota) = \ker(\iota|_{\langle \overline{-1}, \overline{2}\rangle})$. It suffices to study the restriction of $\iota$ to $\langle \overline{-1}, \overline{2} \rangle$. First, the restriction of $\iota$ to $\langle \overline{-1} \rangle$ is injective: if $-1 = \alpha^{2^n}$ for some $\alpha \in k$, then $k$ must contain primitive $2^{n+1}$-roots of unity, which is absurd. Next, we turn to the restriction of $\iota$ to $\langle \overline{2} \rangle$. Note that $\sqrt{2} \in \Q(\zeta_8)$, while $\sqrt[4]{2}$ is not contained in $\Q(\zeta_N)$ since $\Q(\sqrt[4]{2})/\Q$ is not an abelian extension. From this, we deduce that if $N = 4$, then the restriction of $\iota$ to $\langle \overline{2} \rangle$ is injective. We also deduce that if $N \geq 8$, then the kernel of the restriction is $\langle \overline{2^{2^{n-1}}} \rangle$. In conclusion, if $N = 4$, then $\ker \iota$ is trivial and
    \begin{align*}
        \mathcal{W} \simeq \Z/2\Z \times \bigoplus\limits_{\ell \in \mathbb{P}} \Z/4\Z;
    \end{align*}
    if $N = 2^n \geq 8$, then $\ker \iota = \ker(\iota|_{\langle \overline{2} \rangle})$ and
    \begin{align*}
        \mathcal{W} \simeq \Z/2\Z \times \Z/2^{n-1}\Z \times \bigoplus\limits_{\ell \in \mathbb{P} \setminus \{2\}} \Z/2^n\Z. \text{\qed }
    \end{align*} 
\end{enumerate}
\end{proof}

The structure of $\mathcal{W}$ determined, we may describe the corresponding groups $W_i$ of the cyclic fields $K_i$ in combinatorial terms. Note that each $W_i$ is a finite cyclic subgroup of $\mathcal{W}$ for $0\le i\in m$, so we can use only finitely many generators to describe the groups $W_i$. For example, suppose $N = 2^n \geq 8$ and $K = k(a^{1/N})$ for some integer $a \neq 0$. If 
\begin{align*}
    a = (-1)^{n_{-1}} \cdot 2^{n_2} \cdot \prod\limits_{\ell \in \mathbb{P'} \setminus \{2\}} \ell^{n_{\ell}}
\end{align*}
for a finite subset $\bbP'\subset \bbP$,  
then the corresponding finite subgroup is the cyclic subgroup of $\mathcal{W}$ generated by $(\overline{n_{-1}}, \overline{n_2}, ({\ol n_{\ell})_{\ell\in \bbP'\setminus \{2\}}})$.

Using Proposition 3.1, one can compute effectively algebraic patching degrees $\Delta^{\omega}_r$ and algebraic degrees of freedom $f^{\omega}_c$. However, to compute patching degrees $\Delta_r$ and $f_c$, we will need to analyze further the image of a subgroup $W$ in $k_v^{\times}/(k_v^{\times})^{p^n}$. We shall do this when each $K_i$ is in a fixed concrete bicyclic extension in the next section. 



\section{Computing decomposition groups: the case of subfields contained in a bicyclic extension}\label{sec:D}

In the following sections, we shall further restrict to a special case. Fix a prime integer $p$ and a positive integer $n$. Let $k:=\Q(\zeta)$ be the $p^n$-th cyclotomic field, where $\zeta$ is a primitive $p^n$-th root of unity in $\Qbar$, the algebraic closure of $\Q$ in $\C$. We fix an algebraic closure $\Qbar_\ell$ of $\Q_\ell$ and an embedding $\Qbar \embed \Qbar_\ell$. 

Let $\ell_1$ and $\ell_2$ be two distinct prime integers with $\ell_i\neq p$,
and let $F:=k(\alpha_1,\alpha_2)$, where $\alpha_1=\ell_1^{1/p^n}$ and
$\alpha_2=\ell_2^{1/p^n}$. Let $m\ge 1$ be a positive integer. We assume that each component $K_i$ of the étale $k$-algebra $L=\prod_{i=0}^m K_i$ is of the form $K_i=k(\alpha_1^{a_i} \alpha_2^{b_i})$, that is, a cyclic
subextension of $F/k$, where $a_i$ and $b_i$ are integers satisfying $0 \le a_i, b_i < p^n$. The Galois group $G=\Gal(F/k)$ is bicyclic of order 
$p^{2n}$ generated by two elements $\tau_1, \tau_2$,
\[ \tau_1(\alpha_1)=\alpha_1 \zeta, \quad
\tau_1(\alpha_2)=\alpha_2,\quad 
\tau_2(\alpha_1)=\alpha_1,\quad \tau_2(\alpha_2)=\alpha_2 \zeta. \]
Therefore, any subfield of the form $M_c(d)$, which appears in the definition of algebraic degrees of freedom, is automatically a subfield of the bicyclic extension $F$.

\subsection{Decomposition groups and local cyclicity}

Set $F_i =k(\alpha_i)$ for $i=1,2$. Write $G_i=\Gal(F_i/k)$ and we
have a natural isomorphism 
\[ G \isoto G_1\times G_2, \quad \sigma \mapsto (\sigma|_{F_1},
\sigma|_{F_2}). \]
For any prime $\ell$, write $w$, $w_1$, $w_2$, and $v$ for the places
of $F$, $F_1$, $F_2$ and $k$, respectively, lying over $\ell$ with
respect to the embedding $\Qbar \embed \Qbar_\ell$. If $\ell\nmid p
\ell_1 \ell_2$, then $\ell$ is unramified in both $F_1$ and $F_2$ and
hence $\ell$ is
unramified in $F$. Let $G_v$, $G_{1,v}$ and $G_{2,v}$ be the
decomposition groups of $v$ in $G$, $G_1$ and $G_2$, respectively.

For any integers $m_1\neq 0$ and $r$, denote by $[r]_{m_1}$ the residue class of $r$ in $\Z/m_1\Z$. If $m_1$ and $r$ are coprime, let $\ord([r]_{m_1})$ denote the order of $[r]_{m_1}$ in $(\Z/m_1\Z)^\times$. In the following lemma we investigate the ramification after we add a $p^n$-th root of an integer to $\Q_{\ell}(\zeta)$.

\begin{lemma}\label{lm:G1v} Let $\ell \neq p$ be a prime number.

  {\rm (1)} For any positive integer $s$, the field
      extension $\Q_\ell(\zeta, \ell^{s/p^n})/\Q_\ell(\zeta)$ 
      is totally ramified of
      degree $p^{n-v_p(s)}$, where $v_p$ is the normalized
      valuation at $p$. 

  {\rm (2)} For any positive integer $r$ not divisible by $\ell$, 
      the field extension $\Q_\ell(\zeta, r^{1/p^n})/\Q_\ell(\zeta)$
      is unramified of degree 
      \begin{equation}
        \label{eq:f2/f1}
   p^{\max\{\min\{n,s_1\}-(s_1-s_2), 0 \}}     
      \end{equation}
   where $s_1 =v_p(\ell-1)$ and $s_2=v_p(\ord([r]_\ell))$. 
\end{lemma}

\begin{proof}
  (1) We first consider the case where $s=1$. As $\Q_\ell(\zeta, \ell^{1/p^n})$ is the splitting field of the polynomial $f(X) = X^{p^n}-\ell$ over $\Q_\ell(\zeta)$, it suffices to show that $f(X)$ is irreducible. Since $\ell$ is unramified in
      $\Q(\zeta)$, the element $\ell$ is a uniformizer of the complete discrete valuation ring $\Z_\ell[\zeta]$. 
      By Eisenstein's criterion, $f(X)$ is irreducible in $\Z_\ell[\zeta][X]$.
      Therefore, $\Q_\ell(\zeta, \ell^{1/p^n})/\Q_\ell(\zeta)$ is totally ramified of degree $p^n$. 
    
      For general $s$, write $s=p^{v_p(s)} s'$. Then
      $\Q_\ell(\zeta, \ell^{s/p^n}) = \Q_\ell(\zeta, \ell^{s'/p^{n'}})$
      with $n'=n-v_p(s)$. Since $s'$ is prime to $p$, $\Q_\ell(\zeta, \ell^{s'/p^{n'}})=\Q_\ell(\zeta, \ell^{1/p^{n'}})$ is a totally ramified extension of has degree $p^{n'}$ over $\Q_\ell(\zeta)$ of degree $p^{n-v_p(s)}$.

  (2) Since $\ell\nmid p r$, the prime $\ell$ is unramified in
      both $\Q(\zeta)$ and $\Q(r^{1/p^n})$, and therefore 
      $\Q_\ell(\zeta, r^{1/p^n})$ is an unramified extension over 
      $\Q_\ell(\zeta)$. Denote the residue fields of 
      $\Q_\ell(\zeta)$ and $\Q_\ell(\zeta, r^{1/p^n})$ by
      $\F_{\ell^{f_1}}$ and $\F_{\ell^{f_2}}$, respectively. Then
      \[ [\Q_\ell(\zeta, r^{1/p^n}):\Q_\ell(\zeta)]= \frac{f_2}{f_1}. \] 
      We have $f_1=\ord ([\ell]_{p^n})$, the smallest positive integer $f$ such that $\ell^f \equiv 1 \pmod p^n$. 
      Put $s_1=v_p(\ell-1)$, the smallest positive integer $s$ such that $\ell\equiv 1 \pmod p^s$. If $s_1=0$, let $f_0$ be the smallest
      positive integer such that $p$ divides $\ell^{f_0}-1$, then we have
      $f_1=f_0 p^{n-1}$. If $s_1>0$, then $f_1=p^{\min\{n-s_1,0\}}$. 
      
      We know $\F_{\ell^{f_2}}$ is the splitting field of
      the polynomial $f(X)=X^{p^n}-r$ over $\F_\ell$. Let $G$ be the the finite abelian
      group in $\ol\F_\ell^\times$ generated by all roots $\alpha$ of
      $f(X)$. Since $p$ divides the cardinality of $G$, every root $\alpha$ has order $p^n
      \ord([r]_\ell)$ by the fundamental theorem of abelian
      groups. Thus, $f_2$ is the smallest positive integer such
      that $p^n \ord([r]_\ell)$ divides $\ell^{f_2}-1$.
      Put $s_2 = v_p(\ord([r]_\ell))$. If $s_1=0$, then $s_2=0$ and $f_2=f_0 p^{n+s_2-1}=f_0 p^{n-1}$. 
      If $s_1>1$, then 
      $f_2=p^{\min\{n+s_2-s_1,0\}}$. 

      Thus, if $s_1=0$, then $f_2/f_1=1$, If $s_1\ge 1$, then   
      \[ \frac{f_2}{f_1}=
      \begin{cases}
        p^{s_2} & \text{if $s_1\le n$;} \\
        p^{n-(s_1-s_2)} & \text{if $s_1-s_2 \le n \le s_1$;} \\
        1 & \text{if $n\le s_1-s_2$.}  
      \end{cases} \]
      This gives the degree in \eqref{eq:f2/f1}. \qed 
\end{proof}

Now we investigate the structure of the decomposition group $G_v$, where $v$ is a place of $k$ lying over the prime $\ell$.

\begin{lemma}\label{lm:Gv} Let $\ell$ be a prime and $v$ a place of $k$ lying over $\ell$.
Let $G_v$, $G_{1,v}$ and $G_{2,v}$ be the decomposition groups of $v$
in $G$, $G_1$ and $G_2$, respectively.

(1) If $\ell=p$, then $G_v$ is a cyclic group.

(2) If $\ell$ is $\ell_1$ or $\ell_2$, then $G_v\simeq G_{1,v} \times G_{2,v}$. Moreover, if $\ell=\ell_1$, then $G_{1,v}\simeq \Z/p^n\Z$ and $G_{2,v}\simeq \Z/p^{m_{12}}\Z$, where 
\[ 
    m_{12} = \max\left\{ \min\{n,s_1\}-(s_1-s_2), 0 \right\}, \quad
    s_1:=v_p(\ell_1-1), \quad s_2:=v_p(\ord([\ell_2]_{\ell_1})). \]
If $\ell=\ell_2$, then $G_{1,v}\simeq \Z/p^{m_{21}}
    \Z$ and $G_{2,v}\simeq \Z/p^{n}\Z$, where 
\[ m_{21}=\max\{\min\{n,s_1\}-(s_1-s_2), 0 \}, \quad
    s_1:=v_p(\ell_2-1),\quad s_2:=v_p(\ord([\ell_1]_{\ell_2})). \]
\end{lemma}
\begin{proof}
  (1) By Kummer theory, it suffices to show that the group
      $W = \<\ell_1,\ell_2\>$ generated by $\ell_1$ and $\ell_2$ in
      $k_p^\times /(k_p^\times)^{p^n}$ is cyclic. Note that $W$ is a finite $p$-group contained in the image of $\Zp^\times$ and hence in the image of $1+p \Zp$. As a profinite group $1+p\Zp$ is
      isomorphic to $\Zp$, and any finite quotient of $1+p\Zp$ is isomorphic to $(1+p\Zp)/(1+p^{r+1}\Zp)\simeq \Zp/p^r \Zp$ for some $r \ge 0$, which is a cyclic 
      group. Therefore, $W$ is cyclic and $k_p(\alpha_1,\alpha_2)$ is
      a cyclic extension over $k_p$. 

  (2) If $\ell=\ell_1$, then $F_{1,w_1}=\Q_{\ell_1}(\zeta,      
      \alpha_1)$ is totally ramified of degree $p^n$ over $k_v=\Q_{\ell_1}(\zeta)$ and $F_{2,w_2}=\Q_{\ell_1}(\zeta, \alpha_2)$ is unramified of degree $p^{m_{12}}$ over $k_v$ by
      Lemma~\ref{lm:G1v}. Since $F_{1,w_1} F_{2,w_2}=F_w$ and
      $F_{1,w_1}\cap F_{2,w_2}=k_v$, we have $G_v\simeq G_{1,v}\times
      G_{2,v}\simeq \Z/p^n\Z\times \Z/p^{m_{12}} \Z$. Similarly, we
      have the same result  
      for $\ell=\ell_2$. \qed        
\end{proof}

Let $W=\<\ell_1,\ell_2\>$ be the subgroup of $k^\times/(k^\times)^{p^n}$ generated by $\ell_1$ and $\ell_2$. With these generators, we shall write $W= \Z/p^n\Z\times \Z/p^n\Z$. 
Each subfield $K_i=k(\alpha_1^{a_i}\alpha_2^{b_i})\subset F$ then corresponds to the cyclic subgroup of $W$ generated by the element $(a_i,b_i)$. 
Recall that
$[K_i:k]=p^{\epsilon_i}$ and we assume
$\epsilon_0=\min\{\epsilon_i \mid 0 \leq i \leq m \}$. 
We can write $(a_i,b_i)=p^{n-\epsilon_i}(a'_i,b'_i)$ such that $p$ does not divide both $a'_i$ and $b'_i$.
For any subset $c\subset \calI=\{1,\dots, m\}$ and any positive integer
$d\le \min\{\epsilon_i \mid i \in c \}$, the composition field $M_c(d)$ corresponds to the subgroup 
$W_c(d)=p^{n-d}\<(a_i', b_i'): i\in c\>$.

We have identified the Galois group $G=\Gal(F/k)$ with $\Hom(W, \mu)$, where $\mu$ denotes the cyclic group $\langle \zeta \rangle$. For the basis $(1,0), (0,1)$ of $W$, we have a dual basis $\tau_1, \tau_2$ for $\Hom(W, \mu)$:
\[ 
\tau_1((1,0))=\tau_2((0,1))=\zeta, \quad \tau_1((0,1))=\tau_2((1,0))=1. 
\] 
We set $H:=\Gal(M_c(d)/k)$ and write $\pi:G\to H$ for the natural
projection, which can be represented as the restriction map
\[ 
\pi: \Hom(W, \mu) \longrightarrow \Hom(W_c(d),\mu). 
\]
The condition that $M_c(d)/k$ is locally cyclic is
equivalent to that for any finite place $v$ of $k$, the decomposition
group $H_v$ at 
$v$ is cyclic. If $G_v$ is the decomposition group at $v$, then
$H_v=\pi(G_v)$. This provides a method to check whether $M_c(d)$ is locally cyclic.

\begin{lemma}\label{lm:Mcd}
Let $c \subset \calI$ be a subset and $d$ be a positive integer with $d \le \epsilon_0$. Write
\begin{align*}
    W_c(d) = p^{n-d} \langle (c_1(c),d_1(c)),(0,d_2(c)) \rangle
\end{align*}
as a subgroup of $\Z/p^n \Z \times \Z/p^n\Z$ for some $c_1(c),d_1(c), d_2(c)\in \Z/p^n\Z$ using row reduction. Let $m_{12}$ and $m_{21}$ be the integers as in Lemma~\ref{lm:Gv}. Then $M_c(d)$ is locally cyclic if and only if
\begin{align*}
    p^{n-d} c_1(c)\in p^{m_{21}}\Z \quad\text{and}\quad p^{n-d} d_2(c)\in p^{m_{12}}\Z.
\end{align*}  
\end{lemma}

\begin{proof}
Let $\ell$ be a prime integer and $v$ a place of $k$ over $\ell$. If
$\ell\nmid p\ell_1\ell_2$, then $v$ is unramified in $F$ and $G_v$ is
cyclic. If $v=p$, then $G_v$ is cyclic by Lemma~\ref{lm:Gv}. Thus, it
suffices to check the cyclicity of $H_v$ at $\ell=\ell_1$ or $\ell=\ell_2$.   
Let $W_v\subset W$ be the subgroup such that $G_v=\Hom(W/W_v, \mu)$. 
Then $H_v$ is cyclic if and only if the quotient $W_c(d)/W_v$ is cyclic. 

If $\ell=\ell_1$, then $G_v=\Z/p^n\Z\times \Z/p^{m_{12}} \Z$ and
$W_v=\{0\}\times p^{m_{12}}\Z/p^n\Z$ by Lemma~\ref{lm:Gv}. Thus, the quotient group $W_c(d)/W_v$ is cyclic if and only if
\begin{align*}
    p^{n-d} c_1(c)=0 \quad\text{ or }\quad p^{n-d} d_2(c)\equiv 0 \pmod{ p^{m_{12}}}.
\end{align*}
If $\ell=\ell_2$, then $G_v=\Z/p^{m_{21}} \Z\times \Z/p^{n} \Z$ and
$W_v=p^{m_{21}}\Z/p^n\Z \times \{0\}$ by Lemma~\ref{lm:Gv}. Thus, the quotient group  $W_c(d)/W_v$ is cyclic if and only if 
\begin{align*}
    p^{n-d} c_1(c)\equiv 0 \pmod{p^{m_{21}}} \quad\text{ or }\quad p^{n-d} d_2(c)= 0.
\end{align*}
To sum up, $M_c(d)$ is locally cyclic if and only if $p^{n-d} c_1(c)\in p^{m_{21}}\Z$ and $p^{n-d} d_2(c)\in p^{m_{12}}\Z$. \qed  
\end{proof}

\begin{cor}\label{cor:F_loc_cyc}
Let $F=k(\alpha_1,\alpha_2)$ be the bicylic field extension as above. Then $F$ is locally cyclic if and only if 
\[ n \le \min \left\{ v_p(\ell_1-1)-v_p(\ord[\ell_2]_{\ell_1}), v_p(\ell_2-1)-v_p(\ord[\ell_1]_{\ell_2}) \right\} . \]
\end{cor}

\section{Computing Tate-Shafarevich groups and examples}
\label{sec:C}

In view of Sections 5, we have made some assumptions on $k$ and $K_i$. Our aim is to compute the Tate-Shafarevich groups $\Sha^2_{\omega}(k, \widehat{T}_{L/k})$ and $\Sha^2(k, \widehat{T}_{L/k})$ using Theorem \ref{thm:Lee}. We implemented several computer programs that computed all the invariants mentioned in the theorem. The programs use the mathematical software SageMath and can be found on
\begin{center}
    \url{https://github.com/hfy880916/Tate-Shafarevich-groups-of-multinorm-one-torus}.
\end{center}

There are some advantages to make the assumptions above. First, each $K_i$ is contained in the bicyclic extension $k(\sqrt[p^n]{\ell_1}, \sqrt[p^n]{\ell_2})$, so we do not have to check whether a field $M_c(d)$ is a subfield when we compute the algebraic degree of freedom of an equivalence class $c$. Furthermore, the conditions ``$M_c(d)$ is locally cyclic'' and ``$K_0(f)$ is contained in $M_c(d)$'' that appear in the definitions can be converted into problems in finite abelian groups. With these advantages, we can calculate (algebraic) patching degrees and (algebraic) degrees of freedom of examples in reasonable time. Below we illustrate the results by showing two examples

\begin{Example} \label{ex:1}
We put $p = 3$ and $n = 3$, so $k = \Q(\zeta_{27})$. Let the primes $\ell_1 = 5$ and $\ell_2 = 19$. We consider the tori consisting of the following extensions over $k$: $K_0 = k(\sqrt[27]{5})$, $K_1 = k(\sqrt[27]{5 \times 19})$, $K_2 = k(\sqrt[27]{5^2 \times 19^3})$, $K_3 = k(\sqrt[27]{5^3 \times 19^5})$, $K_4 = k(\sqrt[27]{5^5 \times 19^{11}})$. We list $a_i$ and $b_i$ as follows:
\begin{align*}
    a_0 &= 1, & a_1 &= 1, & a_2 &= 2, & a_3 &= 3, & a_4 &= 5, \\
    b_0 &= 0, & b_1 &= 1, & b_2 &= 3, & b_3 &= 5, & b_4 &= 11.
\end{align*}
We see that the $K_i$'s are linearly disjoint. Now we list the $e_{ij}$'s,
\begin{align*}
    [e_{ij}] = \begin{pmatrix}
    3 & 0 & 0 & 0 & 0 \\
    0 & 3 & 0 & 0 & 0 \\
    0 & 0 & 3 & 0 & 0 \\
    0 & 0 & 0 & 3 & 0 \\
    0 & 0 & 0 & 0 & 3
    \end{pmatrix}.
\end{align*}
In this case the only nonempty $U_r$ is $U_0 = \{1, 2, 3, 4\} = \mathcal{I}$, and it has four $1$-equivalence classes $\{1\}$, $\{2\}$, $\{3\}$, $\{4\}$. We compute that $L(U_0) = 0$, the algebraic patching degree $\Delta^{\omega}_0 = 3$, and the patching degree $\Delta_0 = 3$. We compute and list the algebraic degrees of freedom $f^{\omega}_c$ and degrees of freedom $f_c$ for equivalence classes $c = U_0, \{1\}, \{2\}, \{3\}, \{4\}$.
\begin{table}[H]
    \centering
    \begin{tabular}{c|c|c|c|c|c}
        \hline
        \hline
        $c$ & $U_0$ & $\{1\}$ & $\{2\}$ & $\{3\}$ & $\{4\}$ \\
        \hline
        $f^{\omega}_c$ & $3$ & $0$ & $0$ & $0$ & $0$ \\
        $f_c$ & $1$ & NE & NE & NE & NE \\
        \hline
        \hline
    \end{tabular}
    \caption{The algebraic degrees of freedom and degrees of freedom in Example \ref{ex:1}.}
\end{table}
In the table above, ``NE'' stands for ``does not exist''. Using Theorem \ref{thm:Lee} we compute the Tate-Shafarevich groups,
\begin{align*}
    \Sha^2_{\omega}(k, \widehat{T}_{L/k}) &\simeq (\Z/p^{(3-0)}\Z)^{(4-1)} = \Z/27\Z \times \Z/27\Z \times \Z/27\Z; \\
    \Sha^2(k, \widehat{T}_{L/k}) &\simeq (\Z/p^{(1-0)}\Z)^{(4-1)} = \Z/3\Z \times \Z/3\Z \times \Z/3\Z.
\end{align*}
\end{Example}

\begin{Example} \label{ex:2}
Let $p, n, k, \ell_1, \ell_2, m$ be the same. Consider a different tori consisting of $K_0 = k(\sqrt[27]{5})$, $K_1 = k(\sqrt[27]{5 \times 19})$, $K_2 = k(\sqrt[27]{5^2 \times 19^3})$, $K_3 = k(\sqrt[27]{5^4 \times 19^9})$, $K_4 = k(\sqrt[27]{5^{10}\times 19^{19}})$. We list $a_i$ and $b_i$ as follows:
\begin{align*}
    a_0 &= 1, & a_1 &= 1, & a_2 &= 2, & a_3 &= 4, & a_4 &= 10, \\
    b_0 &= 0, & b_1 &= 1, & b_2 &= 3, & b_3 &= 9, & b_4 &= 19.
\end{align*}
The $K_i$'s are no longer linearly disjoint so we expect the components of the Tate-Shafarevich groups to be less regular. We list the $e_{ij}$'s,
\begin{align*}
    [e_{ij}] = \begin{pmatrix}
    3 & 0 & 0 & 1 & 0 \\
    0 & 3 & 0 & 0 & 2 \\
    0 & 0 & 3 & 0 & 0 \\
    1 & 0 & 0 & 3 & 0 \\
    0 & 2 & 0 & 0 & 3
    \end{pmatrix}.
\end{align*}
In this case we have two nonempty $U_r$'s, $U_0 = \{1, 2, 4\}$ and $U_1 = \{3\}$. We present the $\ell$-equivalence relations that need to be considered as follows.
\begin{figure}[H]
\centering
    \begin{subfigure}[t]{0.32\textwidth}
        \centering
        \includegraphics[width=3.5cm]{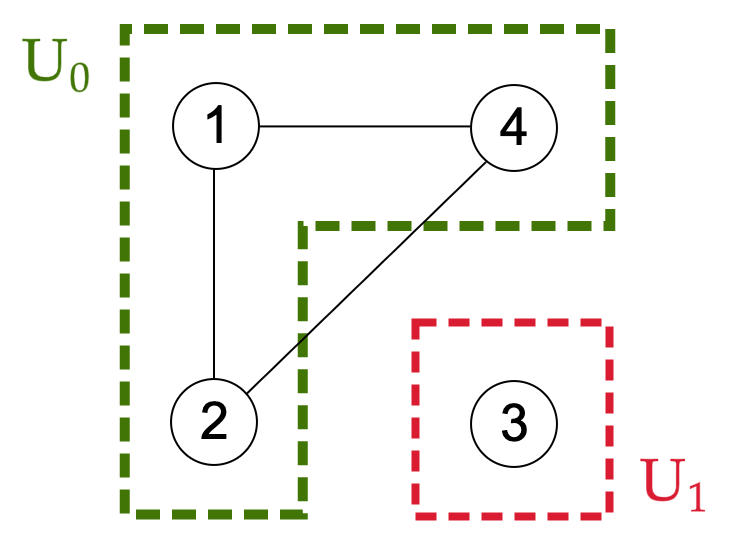}
        \caption{$\ell = 0$}
    \end{subfigure}
    \begin{subfigure}[t]{0.32\textwidth}
        \centering
        \includegraphics[width=3.5cm]{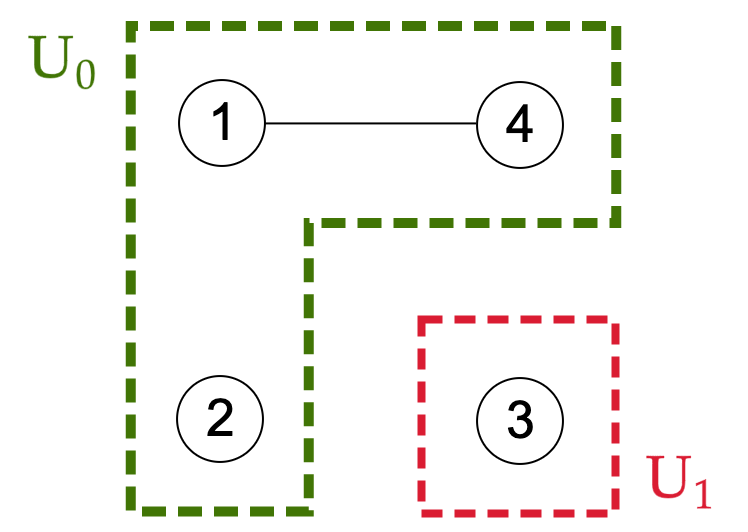}
        \caption{$\ell = 1, 2$}
    \end{subfigure}
    \begin{subfigure}[t]{0.32\textwidth}
        \centering
        \includegraphics[width=3.5cm]{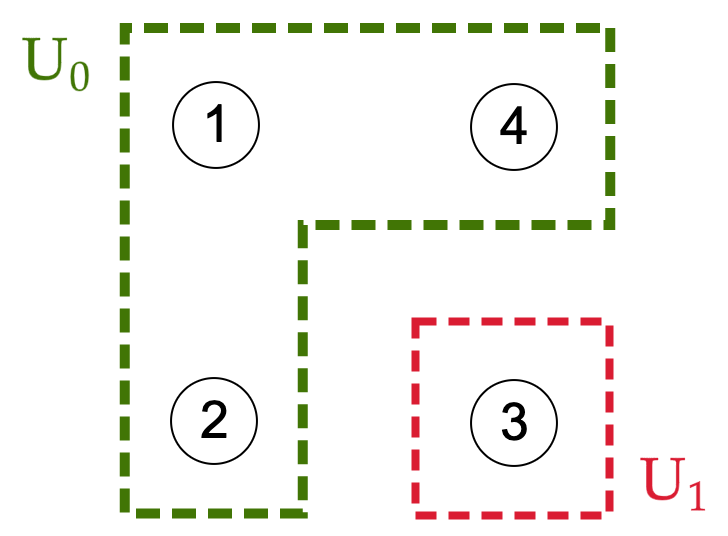}
        \caption{$\ell = 3$}
    \end{subfigure}
\caption{For $i, j \in U_r$, they are connected by a line iff. $i \sim_{\ell} j$.}
\end{figure}
The set $\mathcal{R} = \{0, 1\}$, and we compute that $L(U_0) = 0$, $L(U_1) = 3$. We compute the algebraic patching degrees $\Delta^{\omega}_r$ and patching degrees $\Delta_r$,
\begin{align*}
    \Delta^{\omega}_0 &= 3, & \Delta^{\omega}_1 &= 3, & \Delta_0 &= 1, & \Delta_1 &= 1.
\end{align*}
We compute and list the algebraic degrees of freedom $f^{\omega}_c$ and degrees of freedom $f_c$ for equivalence classes $c = U_0$, $\{1, 4\}$, $\{1\}$, $\{4\}$, $\{2\}$, and $U_1$.
\begin{table}[H]
    \centering
    \begin{tabular}{c|c|c|c|c|c|c}
        \hline
        \hline
        $c$ & $U_0$ & $\{1, 4\}$ & $\{1\}$ & $\{4\}$ & $\{2\}$ & $U_1$ \\
        \hline
        \hline
        $f^{\omega}_c$ & $3$ & $0$ & $0$ & $0$ & $0$ & $1$ \\
        \hline
        $f_c$ & $1$ & NE & NE & NE & NE & NE \\
        \hline
        \hline
    \end{tabular}
    \caption{The algebraic degrees of freedom and degrees of freedom in Example \ref{ex:2}.}
\end{table}
Hence the Tate-Shafarevich groups are
\begin{align*}
    \Sha^2_{\omega}(k, \widehat{T}_{L/k}) &\simeq \Z/p^{3-1}\Z \oplus (\Z/p^{3-0}\Z)^{2-1} =  \Z/9\Z \times \Z/27\Z; \\
    \Sha^2(k, \widehat{T}_{L/k}) &\simeq (\Z/p^{1-0}\Z)^{2-1} = \Z/3\Z.
\end{align*}
\end{Example}

\section*{Acknowledgments}
The authors are grateful to Ting-Yu Lee for her helpful discussions
and generously sharing the ideas of her
work~\cite{lee2022tate}. Thanks also go to Dasheng Wei for his
expertise and helpful discussions.   
The present paper grows up through the authors' participating an
undergraduate research program (URP) held by the National Center for
Theoretical Sciences. They acknowledge the NCTS for the stimulating
environment.  
Liang and Yu are partially supported by the NSTC grant 
109-2115-M-001-002-MY3.



\begin{thebibliography}{10}

\bibitem{BLP19}
E.~Bayer-Fluckiger, T.-Y. Lee, and R.~Parimala.
\newblock Hasse principles for multinorm equations.
\newblock {\em Adv. Math.}, 356:106818, 35, 2019.

\bibitem{demarche-wei}
Cyril Demarche and Dasheng Wei.
\newblock Hasse principle and weak approximation for multinorm equations.
\newblock {\em Israel J. Math.}, 202(1):275--293, 2014.

\bibitem{hurlimann}
W.~H\"{u}rlimann.
\newblock On algebraic tori of norm type.
\newblock {\em Comment. Math. Helv.}, 59(4):539--549, 1984.

\bibitem{lee2022tate}
T-Y Lee.
\newblock The {T}ate-{S}hafarevich groups of multinorm-one tori.
\newblock {\em Journal of Pure and Applied Algebra}, 226(7):106906, 2022.

\bibitem{NSW}
J\"{u}rgen Neukirch, Alexander Schmidt, and Kay Wingberg.
\newblock {\em Cohomology of number fields}, volume 323 of {\em Grundlehren der
  Mathematischen Wissenschaften [Fundamental Principles of Mathematical
  Sciences]}.
\newblock Springer-Verlag, Berlin, 2000.

\bibitem{ono:tamagawa}
Takashi Ono.
\newblock On the {T}amagawa number of algebraic tori.
\newblock {\em Ann. of Math. (2)}, 78:47--73, 1963.

\bibitem{platonov-rapinchuk}
Vladimir Platonov and Andrei Rapinchuk.
\newblock {\em Algebraic groups and number theory}, volume 139 of {\em Pure and
  Applied Mathematics}.
\newblock Academic Press, Inc., Boston, MA, 1994.
\newblock Translated from the 1991 Russian original by Rachel Rowen.

\bibitem{pollio}
Timothy~P. Pollio.
\newblock On the multinorm principle for finite abelian extensions.
\newblock {\em Pure Appl. Math. Q.}, 10(3):547--566, 2014.

\bibitem{pollio-rapinchuk}
Timothy~P. Pollio and Andrei~S. Rapinchuk.
\newblock The multinorm principle for linearly disjoint {G}alois extensions.
\newblock {\em J. Number Theory}, 133(2):802--821, 2013.

\bibitem{prasad-rapinchuk}
Gopal Prasad and Andrei~S. Rapinchuk.
\newblock Local-global principles for embedding of fields with involution into
  simple algebras with involution.
\newblock {\em Comment. Math. Helv.}, 85(3):583--645, 2010.

\end{thebibliography}

\def\cprime{$'$}

\end{document}